\documentclass[envcountsame]{llncs}
\usepackage{amsmath,amssymb}
\usepackage{url}
\usepackage{mathrsfs}
\usepackage{verbatim}
\usepackage{graphics}
\usepackage{color}

\newcommand{\isom}{\cong}
\newcommand{\FF}{\mathbb{F}}
\newcommand{\PP}{\mathbb{P}}

\newcommand{\ZZ}{\mathbf{Z}}
\newcommand{\oO}{\mathrm{O}}

\newcommand{\cK}{\mathscr{K}}

\newcommand{\bbmu}{\mu\!\!\!\!\mu}
\newcommand{\Mul}{\mathbf{M}}
\newcommand{\mul}{\mathbf{m}}
\newcommand{\Sqr}{\mathbf{S}}
\newcommand{\wx}{\omega}
\newcommand{\wy}{\overline{\wx}}
\newcommand{\ccomma}{\raisebox{0.4ex}{,}}
\newcommand{\ignore}[1]{}
\newcommand{\exclude}[1]{}

\title{
Efficient arithmetic on elliptic curves\\in characteristic 2}
\author{David Kohel}
\institute{
Institut de Math{\'e}matiques de Luminy\\
Universit{\'e} d'Aix-Marseille\\
163, avenue de Luminy, Case 907\\
13288 Marseille Cedex 9\\ France}

\begin{document}

\maketitle

\begin{abstract}
We present normal forms for elliptic curves over a field of characteristic~$2$ 
analogous to Edwards normal form, and determine bases of addition laws, which 
provide strikingly simple expressions for the group law.  
We deduce efficient algorithms for point addition and scalar multiplication 
on these forms.  The resulting algorithms apply to any elliptic curve over a 
field of characteristic~2 with a 4-torsion point, via 
an isomorphism with one of the normal forms.
We deduce algorithms for 
duplication in time $2\Mul + 5\Sqr + 2\mul_c$  
and for addition of points in time $7\Mul + 2\Sqr$, 
where $\Mul$ is the cost of multiplication, $\Sqr$ the cost of squaring, and 
$\mul_c$ the cost of multiplication by a constant. 
By a study of the Kummer curves $\cK = E/\{[\pm1]\}$, we develop 
an algorithm for scalar multiplication with point recovery 
which computes the multiple of a point $P$ with 
$4\Mul + 4\Sqr + 2\mul_c + \mul_t$ per bit where $\mul_t$ is 
multiplication by a constant that depends on $P$.
\end{abstract}

\section{Introduction}

The last five years have seen significant improvements in the efficiency 
of known algorithms for arithmetic on elliptic curves, spurred by the 
introduction of the Edwards model~\cite{Edwards} and its 
analysis~\cite{BernsteinLange-Edwards,BernsteinEtAl-TwistedEdwards,Hisil-EdwardsRevisited}.  
Previously, it had been recognized that alternative models of elliptic 
curves could admit efficient arithmetic~\cite{ChudnovskyBrothers}, 
but the fastest algorithms could be represented in terms of functions 
on elliptic curves embedded in $\PP^2$ as Weierstrass models.  

Among the best alternative models one finds a common property of symmetry.
They admit a large number of (projective) linear automorphisms, often 
given by signed or scaled coordinate permutations.  An elliptic curve with 
$j$-invariant $j \ne 0, 12^3$ admits only $\{[\pm1]\}$ as automorphism group 
{\it fixing the identity element}.
However, as a genus 1 curve, it also admits translations by rational points, 
and a translation morphism $\tau_Q(P) = P + Q$ on $E$ is projectively linear,
i.e.\ induced by a linear transformation of the ambient projective space, 
if and only if $E$ is a degree $n$ model determined by a complete linear 
system in $\PP^{n-1}$ and $Q$ is in the $n$-torsion subgroup.  As a consequence 
the principal models of cryptographic interest are elliptic curves in $\PP^2$ 
with rational $3$-torsion points (e.g.\ the Hessian models) and in $\PP^3$ with 
$2$-torsion or $4$-torsion points (e.g.\ the Jacobi quadratic intersections 
and Edwards model), and unfortunately, the latter models do not have good 
reduction to characteristic $2$.  The present work aims to fill this gap.

A rough combinatorial explanation for the role of symmetry in efficiency 
is the following.  Suppose that the sum of $x = (x_0:\dots:x_r)$ and 
$y = (y_0:\dots:y_r)$  is expressed by polynomials $(p_0(x,y):\dots:p_r(x,y))$ 
of low bidegree, say $(2,2)$, in $x_i$ and $y_j$.  Such polynomials form a 
finite dimensional space. A translation morphism $\tau$ given by scaled 
coordinate transformation on $E$ determines a new tuple
$
(p_0(\tau(x),\tau^{-1}(y)):\dots,p_r(\tau(x),\tau^{-1}(y))).
$
If $(p_0(x,y):\dots:p_r(x,y))$ is an eigenvector for this transformation
then it tends to have few monomials.  In the case of Hessian, Jacobi, 
Edwards, and similar models, there exist bases of eigenvector polynomial 
addition laws such that the $p_j$ achieve the minimal value of two terms. 

Section~\ref{sec:Symmetries} recalls several results, observations,
and conclusions of Kohel~\cite{Kohel} on symmetries of elliptic curves 
in their embeddings.
As illustration, Section~\ref{sec:EdwardsModel} recalls the main 
properties of the Edwards model as introduced by Edwards~\cite{Edwards}, 
reformulated by Bernstein and Lange~\cite{BernsteinLange-Edwards} 
with twists by Bernstein et al.~\cite{BernsteinEtAl-TwistedEdwards}, 
and properties of its arithmetic described in Hisil et 
al.~\cite{Hisil-EdwardsRevisited} and Bernstein and 
Lange~\cite{BernsteinLange-Completed}.

This background motivates the introduction and classification of 
new models for elliptic curves in Section~\ref{sec:AxiomsD4}, 
based on imitation of the desired properties of Edwards curves, and 
in Section~\ref{sec:NormalForms} we present new elliptic curve models,  
the $\ZZ/4\ZZ$-normal form and the split $\bbmu_4$-normal form, 
which satisfy these properties.  
In Section~\ref{sec:Isomorphisms} we classify all symmetric quartic 
elliptic curves in $\PP^3$ with a rational $4$-torsion point, 
up to projective linear isomorphism.\!\footnote{%
Note that any quartic plane model has a canonical extension to a 
nonsingular quartic model in $\PP^3$ by extending to a complete linear system.}  
In particular we prove that any such curve is linearly isomorphic 
to one of these two models. In Section~\ref{sec:AdditionLaws} we determine 
the polynomial addition laws and resulting complexity for arithmetic 
on these forms.  Finally Section~\ref{sec:KummerCurves} develops 
models for the Kummer curve $\cK = E/\{[\pm1]\}$ and exploits an 
embedding of $E$ in $\cK^2$ in order to develop a Montgomery ladder 
for scalar multiplication with point recovery.  
Section~\ref{sec:Conclusion} summarizes the new complexity results 
for these models in comparison with previously known models and 
algorithms.  An appendix gives the addition laws for a descended 
$\bbmu_4$-normal form that allows us to save on multiplications by 
constants involved in the curve equation.

\vspace{2mm}
\noindent{\bf Notation.}

In what follows we use $\Mul$ and $\Sqr$ for the complexity of multiplication 
and squaring, respectively, in the field $k$, and $\mul_c$ for a multiplication 
by a fixed (possibly small) constant $c$ (or constants $c_i$).\!\footnote{%  
When the small constant is a bounded power of a fixed constant we omit 
the squarings or products entailed in its construction and continue to 
consider $c^{O(1)}$ a fixed constant.}  
For the purposes of complexity analysis we ignore field additions.

When describing a morphism $\varphi:X \rightarrow Y$ given by polynomial maps, 
we write 
$$
\varphi(x) = 
\left\{
\begin{array}{c}
\big(p_{1,0}(x):\dots:p_{1,n}(x)\big),\\
\vdots \\
\big(p_{m,0}(x):\dots:p_{m,n}(x)\big),\\
\end{array}
\right.
$$
to indicate that each of the tuples of polynomials 
$(p_{i,0}(x),\dots,p_{i,n}(x))$ defines the morphism 
on an open neighborhood $U_i \subset X$, namely on 
the complement of the common zeros $p_{i,0}(x) 
= \cdots p_{i,m}(x) = 0$, that any two agree on 
the intersections $U_i \cap U_j$, and that the 
union of the $U_i$ is all of $X$. 

For the projective coordinate functions on $\PP^r$, 
with $r > 3$, we use $X_i$ and so $x = (X_0:\dots:X_r)$ 
represents a generic point.
We also use $X_i$ for their restriction to a curve $E$, 
in which case the $X_i$ are defined modulo the defining 
ideal of $E$.  In the product $\PP^r \times \PP^r$, we 
continue to write $x$ for the first coordinate and use 
$(x,y)$ for a generic point in $\PP^r \times \PP^r$, 
where $y = (Y_0:\dots:Y_r)$.

\section{Elliptic curves with symmetries}
\label{sec:Symmetries}

We consider conditions for an elliptic curve embedding in 
$\PP^r$ to admit many projective linear transformations, 
or symmetries.  In what follows, we recall standard 
definitions and conclusions drawn from Kohel~\cite{Kohel}
(reformulated here without the language of invertible 
sheaves).  The examples of Hessian curves and Edwards 
curves\footnote{%
In particular my discussions of symmetries with Bernstein 
and Lange motivated a study of symmetries in the unpublished 
work~\cite{BernsteinKohelLange} (see the EFD~\cite{EFD}) 
on twisted Hessian curves, picked up by Joye and Rezaeian 
Farashahi~\cite{JoyeReza} after posting to the EFD).
This further led the author to develop a general framework 
for symmetries and to classify the linear action of torsion
in~\cite{Kohel}.}
play a pivotal role in motivating~\cite{Kohel} and 
further examples 
(see Bernstein and Lange~\cite{BernsteinLange-Completed},
Joye and Rezaeian Farashahi~\cite{JoyeReza}, 
Kohel~\cite[Section 8]{Kohel})
suggest that such symmetries go hand-in-hand 
with efficient forms for their arithmetic.\!\footnote{%
By efficient forms, we mean sparse polynomials 
expressions with small coefficients.  These may or may 
not yield the most efficient {\it algorithms}, as seen 
in comparing the evaluation of similarly sparse addition 
laws for the Edwards and $\ZZ/4\ZZ$-normal forms.}

The automorphism group of an elliptic curve $E$ is a finite group, 
and if $j(E) \ne 0, 12^3$, this group is $\{\,[\pm1]\,\}$. 
Inspection of standard projective models for elliptic curves 
shows that the symmetry group can be much greater.
The disparity is explained by the existence of subgroups of 
rational torsion. The automorphism group of an elliptic 
curve is defined to be the automorphisms of the curve which 
fix the identity point, which does not include translations. 
For any rational torsion point $T$, the translation-by-$T$ 
map $\tau_T$ is an automorphism of the curve, which may 
give rise to the additional symmetries. 

We restrict to models of elliptic curves given by complete 
linear systems of a given degree~$d$.  Basically, such a 
curve is defined by $E \subset \PP^r$ such that $r = d-1$, 
$E$ is not contained in any hyperplane, and any hyperplane 
$H$ intersects $E$ in exactly $d$ points, counted with 
multiplicities.  For embedding degree~$3$, such a curve is 
given by a single homogeneous form $F(X,Y,Z)$ of degree~$3$, 
and for degree~$4$ we have an intersection of two quadrics 
in $\PP^3$.  Quartic plane models formally lie outside of 
this scope --- they are neither nonsingular nor given by a 
complete linear system --- but determine a unique degree~4 
elliptic curve in $\PP^3$ after completing the basis of 
functions. As in the case of the Edwards curve, we 
always pass to this model to apply the theory. 

\begin{definition}
\label{def:symmetric_embedding}
Let $E \subset \PP^r$ be an elliptic curve embedded with 
respect to a complete linear system.  We say that $E$ is 
a symmetric model if $[-1]$ is induced by a projective 
linear transformation of $\PP^r$. 
\end{definition}

We next recall a classification of symmetric embeddings 
of elliptic curves (cf.~Kohel~\cite[Lemma 2]{Kohel} for 
the statement in terms of invertible sheaves). 

\begin{lemma}
\label{lem:divisor_class}
Let $E \subset \PP^r$ be an elliptic curve over $k$ embedded with 
respect to a complete linear system.  There exists a point $S$ 
in $E(k)$ such that for any hyperplane $H$ in $\PP^r$ not 
containing $E$, the set of 
points in the intersection $E \cap H = \{ P_0,\dots,P_r \}$, 
in $E(\bar{k})$, counted with multiplicity, sum to $S$.  
The model is symmetric if and only if $S$ is in the subgroup 
$E[2]$ of $2$-torsion points. 
\end{lemma}

\begin{definition}
\label{def:embedding_divisor}
Let $E$ be a degree $d$ embedding in $\PP^r$ with respect to 
a complete linear system, and let $S$ be the point as in the 
previous lemma.  We define the {\it embedding divisor class} 
of $E$ to be $(d-1)(\oO) + (S)$. 
\end{definition}

We describe here the classification of elliptic curves with 
projective embedding, up to {\it linear} isomorphism, rather 
than isomorphism.\!\footnote{%
In recent cryptographic literature, there has been a trend 
to refer to existence of a {\it birational equivalence}. 
In the context of elliptic curves, by definition nonsingular 
projective curves, this concept coincides with isomorphism, 
and we want to identity the subclass of isomorphisms which 
are linear with respect to the coordinate functions of the 
given embedding.}
The notion of isomorphisms given by linear transformations plays 
an important role in the addition laws, since such a change 
of variables gives an isomorphism between the respective spaces 
of addition laws of fixed bidegree $(m,n)$, as described in 
Kohel~\cite[Section 7]{Kohel}.
For a point $T$, we denote the translation-by-$T$ morphism 
by $\tau_T$, given by $\tau_T(P) = P + T$. We now recall 
the classification of symmetries which arise from the 
group law~\cite[Lemma 5]{Kohel}. 

\begin{lemma}
\label{lem:translation_linear}
Let $E \subset \PP^r$ be embedded with respect to the complete 
linear system of degree $d$ and let $T$ be in $E(\bar{k})$.
The translation-by-$T$ morphism is induced by a projective 
linear automorphism of $\PP^r$ if and only if $dT = \oO$.
\end{lemma}

Similarly, we recall the classification of projective linear 
isomorphisms between curves in $\PP^r$ 
(see Kohel~\cite[Lemma 3]{Kohel} for a slightly stronger formulation). 

\begin{lemma}
\label{lem:embedding_class}
Let $E_1$ and $E_2$ be isomorphic elliptic curves embedded 
in $\PP^r$ with respect to complete linear systems of the 
same degree $d$. An isomorphism $\varphi: E_1 \rightarrow E_2$ 
is induced by a projective linear transformation if and 
only if $\varphi(S_1) = S_2$, where $S_i \in E_i(k)$ 
determine the embedding divisor classes $(d-1)(\oO) + (S_i)$ 
of the embeddings.
\end{lemma}

\noindent{\bf Remark.}
By definition, an isomorphism $\varphi: E_1 \rightarrow E_2$ 
of elliptic curves takes the identity of $E_1$ to the 
identity of $E_2$.  It may be possible to define a projective 
linear transformation from $E_1$ to $E_2$ which does not 
respect the group identities (hence is not a group isomorphism). 

\section{Properties of the Edwards normal form}
\label{sec:EdwardsModel}

In this section we suppose that $k$ is a field of characteristic 
different from~2. 
To illustrate the symmetry properties of the previous section 
and motivate the analogous construction in characteristic~2, 
we recall the principal properties of the Edwards normal form, 
summarizing work of Edwards~\cite{Edwards}, 
Hisil et al.~\cite{Hisil-EdwardsRevisited}, and 
Bernstein and Lange~\cite{BernsteinLange-Completed}.
We follow the definitions and notation of Kohel~\cite{Kohel},
defining the twisted Edwards normal form $E/k$ in $\PP^3$: 
\begin{equation*}
\label{eqn:EdwardsModel}
c X_1^2 + X_2^2 = X_0^2 + d X_3^2,\ X_0 X_3 = X_1 X_2,\ \oO = (1:0:1:0). 
\end{equation*}

\subsection*{Edwards model for elliptic curves}

In 2007, Edwards introduced a new model for elliptic curves~\cite{Edwards}, 
defined by the affine model 
$$
x^2 + y^2 = a^2(1 + z^2),\ z = xy,
$$ 
over any field $k$ of characteristic different from $2$.  
The complete linear system associated to this degree~$4$ model has basis 
$\{1,x,y,z\}$ such that the image $(1:x:y:z)$ is a nonsingular projective 
model in $\PP^3$:
$$
X_1^2 + X_2^2 = a^2 (X_0^2 + X_3^2),\ X_0 X_3 = X_1 X_2,
$$
with identity $\oO = (a:0:1:0)$, as a family of curves over $k(a)$. % = k(X(4))$.  
We hereafter refer to this model as the split Edwards model. 
Bernstein and Lange~\cite{BernsteinLange-Edwards} introduced a rescaling 
to descend to $k(d) = k(a^4)$, and subsequently (with Joye, 
Birkner, and Peters~\cite{BernsteinEtAl-TwistedEdwards}) a quadratic twist 
by $c$, to define the twisted Edwards model with $\oO = (1:0:1:0)$:
$$
c X_1^2 + X_2^2 = X_0^2 + d X_3^2,\ X_0 X_3 = X_1 X_2.
$$
The twisted Edwards model in this form appears in Hisil et 
al.~\cite{Hisil-EdwardsRevisited} (as extended Edwards coordinates), 
which provides the most efficient arithmetic. We next recall the 
principal properties of the Edwards normal form (with $c = 1$). 
\vspace{2mm}

\subsection*{Symmetry properties}

\begin{enumerate}
\item
The embedding divisor class is $3(\oO) + (S)$ where $S = 2T$. %is the $2$-torsion point $S = (1:0:-1:0)$. 
\item 
The point $T = (1:-1:0:0)$ is a rational $4$-torsion point. 
\item
The translation--by--$T$ and inverse morphisms are given by: 
$$
\begin{array}{r}
\tau_T(X_0:X_1:X_2:X_3) = (X_0:-X_2:X_1:-X_3),\\
{[-1]}(X_0:X_1:X_2:X_3) = (X_0:-X_1:X_2:-X_3). 
\end{array}
$$
\item
The model admits a factorization $s \circ (\pi_1 \times \pi_2)$ 
through $\PP^1 \times \PP^1$, where 
$$
\hspace{-4mm}
\pi_1(X_0:X_1:X_2:X_3) = 
\left\{
\begin{array}{l}
(X_0:X_1), \\
(X_2:X_3)
\end{array}
\right.
\ccomma\ 
\pi_2(X_0:X_1:X_2:X_3) = 
\left\{
\begin{array}{l}
(X_0:X_2), \\
(X_1:X_3).
\end{array}
\right.
$$
and $s$ is the Segre embedding
$$
s((U_0:U_1),(V_0:V_1)) = (U_0 V_0 : U_1 V_0 : U_0 V_1 : U_1 V_1).
$$
\end{enumerate}
{\bf Remark.} The linear expression for $[-1]$ implies that 
the embedding is symmetric.
This linearity is a consequence of the form of the embedding divisor 
$3(\oO)+(S)$, in view of Lemma~\ref{lem:divisor_class}. 
In addition the two projections are symmetric, in the sense that they are 
stable under $[-1]$.  This is due to the fact that the divisors $2(O) = 
\pi_1^*(\infty)$ and $(O)+(T) = \pi_2^*(\infty)$ are symmetric.

\subsection*{A remarkable factorization}

Hisil et al.~\cite{Hisil-EdwardsRevisited} discovered amazingly 
simple bilinear rational expressions for the affine addition laws, 
which can be described as a factorization of the addition laws 
through the isomorphic curve in $\PP^1 \times \PP^1$ (see 
Bernstein and Lange~\cite{BernsteinLange-Completed} for further 
properties). As a consequence of the symmetry of the embedding 
and its projections, the composition of the addition morphism
$
\mu : E \times E \longrightarrow E
$
with each of the projections $\pi_i: E \rightarrow \PP^1$ admits a basis 
of {\it bilinear} defining polynomials. For $\pi_1 \circ \mu$ 
$\pi_1 \circ \mu$, respectively, we have 
\begin{center}
\scalebox{0.96}{
$
\left\{ 
\begin{array}{l}
  ( X_0 Y_0 + d X_3 Y_3,\; X_1 Y_2 + X_2 Y_1 ), \\
  ( c X_1 Y_1 + X_2 Y_2,\; X_0 Y_3 + X_3 Y_0 )
\end{array} 
\right\}
\mbox{ and }
\left\{ 
\begin{array}{l}
( X_1 Y_2 - X_2 Y_1,\; -X_0 Y_3 + X_3 Y_0 ), \\
( X_0 Y_0 - d X_3 Y_3,\; -c X_1 Y_1 + X_2 Y_2 )
\end{array}
\right\}\cdot
$
}
\end{center}
Addition laws given by polynomial maps of bidegree $(2,2)$ are 
recovered by composing with the Segre embedding.  
This factorization led the author to prove dimension formulas 
for these addition law projections and classify the exceptional 
divisors~\cite{Kohel}. In particular, this permits 
one to prove {\it a priori} the form of the exceptional divisors 
described in Bernstein and Lange~\cite[Section 8]{BernsteinLange-Completed},  
show that these addition laws span all possible addition laws 
of the given bidegree, and conclude their completeness.  

\section{Axioms for a $D_4$-linear model}
\label{sec:AxiomsD4}

The previous sections motivate the study of symmetric quartic models 
of elliptic curves with a rational $4$-torsion point $T$.  
For such a model, we obtain a $4$-dimensional linear representation of 
$
D_4 \isom \langle [-1] \rangle \ltimes \langle \tau_T \rangle,
$
induced by the action on the linear automorphisms of $\PP^3$.
Here we give characterizations of elliptic curve models for which 
this representation is given by coordinate permutation. 

Suppose that $E/k$ is an elliptic curve with $\mathrm{char}(k) = 2$ and $T$ a 
rational $4$-torsion point.  In view of the previous lemmas and the properties 
of Edwards' normal form, we consider reasonable hypotheses for a characteristic~$2$ 
analog.  We note that in the Edwards model, $\tau_T$ acts by signed coordinate 
permutation, which we replace with a permutation action in characteristic $2$.  
\begin{enumerate}
\item
\label{cond1}
The embedding of $E \rightarrow \PP^3$ is a quadratic intersection.
\item
\label{cond2}
$E$ has a rational $4$-torsion point $T$.
\item
\label{cond3}
The group $\langle [-1]\rangle \ltimes \langle \tau_T \rangle \isom D_4$ 
acts by coordinate permutation, and in particular 
$
\tau_T(X_0:X_1:X_2:X_3) = (X_3:X_0:X_1:X_2).
$
\item
\label{cond4}
There exists a symmetric factorization of $E$ through $\PP^1 \times \PP^1$.
\end{enumerate}
Combining conditions~\ref{cond3} and \ref{cond4}, we assume that $E$ lies 
in the skew-Segre image $X_0 X_2 = X_1 X_3$ of $\PP^1 \times \PP^1$.
In order for the representation of $\tau_T$ to stabilize the image of 
$\PP^1 \times \PP^1$, we have
$$
\PP^1 \times \PP^1 \longrightarrow \PP^3,
$$
whose image is $X_0 X_2 = X_1 X_3$, in isomorphism with $\PP^1 \times \PP^1$
by the projections
$$
\pi_1(X_0:X_1:X_2:X_3) = 
\left\{
\begin{array}{l}
(X_0:X_1), \\
(X_3:X_2),
\end{array}
\right.
\ccomma\ 
\pi_2(X_0:X_1:X_2:X_3) = 
\left\{
\begin{array}{l}
(X_0:X_3), \\
(X_1:X_2).
\end{array}
\right.
$$
Secondly, up to isomorphism, there are {\it two} permutation representations 
of $D_4$, both having the same image.  The two representations are distinguished 
by the image of $[-1]$, up to coordinate permutation, being one of the two
$$
\begin{array}{l}
{[-1]}(X_0:X_1:X_2:X_3) = (X_3:X_2:X_1:X_0) \mbox{ or } (X_0:X_3:X_2:X_1).
\end{array}
$$
Considering the form of the projection morphisms $\pi_1$ and $\pi_2$, we see 
that only the first of the possible actions of $[-1]$ stabilizes $\pi_1$ and 
$\pi_2$, while the second exchanges them.  In the next section we are able 
to write down a normal form with $D_4$-permutation action associated to each 
of the possible actions of $[-1]$. 

\section{Normal forms}
\label{sec:NormalForms}

The objective of this section is to introduce elliptic curve models 
which satisfy the desired axioms of the previous section.  
After their definition we list their main properties, whose proof 
is essentially immediate from the symmetry properties of the model. 
We first present the objects of study over a general field $k$ 
before passing to $k$ of characteristic~$2$.
Additional details of their construction can be found in the talk 
notes~\cite{AGCT2011} where they were first introduced.

\begin{definition}
\label{def:Z4-NormalForm}
An elliptic curve $E/k$ in $\PP^3$ is said to be in 
{\it $\ZZ/4\ZZ$-normal form} if it is given by the equations 
$$
X_0^2 - X_1^2 + X_2^2 - X_3^2 = e X_0 X_2 = e X_1 X_3,
$$
with identity $\oO = (1:0:0:1)$.
\end{definition}

The $\ZZ/4\ZZ$-normal form is the unique model, up to linear 
isomorphism (see Theorem~\ref{thm:normal_form_isomorphism}), 
satisfying the complete set of axioms of the previous section.
If we drop the condition for the factorization through 
$\PP^1 \times \PP^1$ (condition~\ref{cond4}), we obtain 
the following normal form, which admits the alternative 
action of $[-1]$. 

\begin{definition}
\label{def:Mu4-NormalForm}
An elliptic curve $C/k$ in $\PP^3$ is said to be in 
{\it split $\bbmu_4$-normal form} if it is given by the equations 
$$
X_0^2-X_2^2 = c^2\,X_1 X_3,\; 
X_1^2-X_3^2 = c^2\,X_0 X_2,
$$
with identity $\oO = (c:1:0:1)$.  
\end{definition}

These normal forms both have good reduction in characteristic~$2$.
The $\ZZ/4\ZZ$-normal form admits a rational $4$-torsion point 
$T = (1:1:0:0)$, and the isomorphism 
$$
\langle T \rangle
  = \{ (1:0:0:1), (1:1:0:0), (0:1:1:0), (0:0:1:1) \}
\isom \ZZ/4\ZZ 
$$
gives the name to curves in this form. 

On the split $\bbmu_4$-normal form, the point $T = (1:c:1:0)$ is 
a rational $4$-torsion point, and if $\mathrm{char}(k) \ne 2$ and 
there exists a primitive $4$-th root of unity $i$ in $k$, then 
$R = (c:i:0,-i)$ is a rational $4$-torsion point (dual to $T$ 
under the Weil pairing) such that $\langle T, R \rangle = C[4]$. 
The subgroup 
$$
\langle R \rangle = 
  \{ (c:1:0:1), (c:i:0:-i), (c:-1:0:-1), (c:-i:0:i) \}
\isom \bbmu_4 
$$
is a group (scheme) isomorphic to the group (scheme) $\bbmu_4$ of 
$4$-th roots of unity, which gives the name to this normal form.  
The nonsplit variant (see Remark following Corollary~\ref{cor:mu4-duplication-complexity})
descends to any subfield containing $c^4$, does not necessarily 
have a rational $4$-torsion point, but in the application to 
elliptic curves over finite fields of characteristic~2, every 
such model can be put in the split form. 
The action of the respective points $T$ by translation gives 
the coordinate permutation action which we desire, the dual 
subgroup $\langle{R}\rangle$ degenerates in characteristic~$2$ 
to the identity group $\{\oO\} = \{(c:1:0:1)\}$, and the 
embedding divisor $3(\oO)+(S)$, where $S = 2R$, degenerates to $4(\oO)$. 
Hereafter we consider these models only over a field of characteristic~2.

\begin{comment}
The $\bbmu_4$-normal form so defined is the previously defined level~4 
canonical model~\cite{Kohel}, up to scalar renormalization of the 
variables to have good reduction at $2$.  In characteristic different 
to $2$ this model is isomorphic to a twist by $-1$ of an Edwards curve.  
In the application to curves over finite fields of characteristic~2, 
since any element is a $4$-th power, we may put any such model in 
split $\bbmu_4$-normal form, so we give preference to the split 
normal form.
\end{comment}

We now formally state and prove the main symmetry properties of 
the new models over a field of characteristic~$2$ with analogy 
to the Edwards model.

\begin{theorem}%[Properties of the $\ZZ/4\ZZ$-normal form]
\label{thm:Z4Symmetries}
Let $E/k$ be a curve in $\ZZ/4\ZZ$-normal form over a field of characteristic~$2$.
\begin{enumerate}
\item
The embedding divisor class is $3(\oO) + (S)$ where $S = (0:1:1:0) = 2T$.
\item
The point $T = (1:1:0:0)$ is a rational $4$-torsion point.
\item
The translation--by--$T$ and inverse morphisms are given by:
$$
\begin{array}{r}
\tau_T(X_0:X_1:X_2:X_3) = (X_3:X_0:X_1:X_2),\\
{[-1]}(X_0:X_1:X_2:X_3) = (X_3:X_2:X_1:X_0).
\end{array}
$$
\item
$E$ admits a factorization through $\PP^1 \times \PP^1$, where
$$
\hspace{-4mm}
\pi_1(X_0:X_1:X_2:X_3) = 
\left\{
\begin{array}{l}
(X_0:X_1),\\
(X_3:X_2),
\end{array}
\right.
\ccomma\ 
\pi_2(X_0:X_1:X_2:X_3) = 
\left\{
\begin{array}{l}
(X_0:X_3),\\
(X_1:X_2).
\end{array}
\right.
$$
\end{enumerate}
\noindent
More precisely, if $(U_0,U_1)$ and $(V_0,V_1)$ are the coordinate functions 
on $\PP^1 \times \PP^1$, the product morphism $\pi_1 \times \pi_2$ determines 
an isomorphism $E \rightarrow E_1$, where $E_1$ is the curve
$
(U_0 + U_1)^2(V_0 + V_1)^2 = c\,U_0 U_1 V_0 V_1,
$ 
whose inverse is the 
restriction of the skew-Segre embedding
$
((U_0:U_1),(V_0:V_1)) \longrightarrow (U_0 V_0 : U_1 V_0 : U_1 V_1: U_0 V_1).
$
\end{theorem}

\begin{proof}
The correctness of the forms for $[-1]$ and $\tau_T$ follow 
from the fact that they are automorphisms, that the asserted 
map for $[-1]$ fixes $\oO$ and that for $\tau_T$ has no fixed 
point, and that $\tau_T(\oO) = T$. 
Since $\tau_T^4 = 1$, it follows that $T$ is $4$-torsion. 
The hypersurface $X_0+X_1+X_2+X_3=0$ cuts out the subgroup 
$\langle{T}\rangle \isom \ZZ/4\ZZ$, which determines the 
embedding divisor class as $3(\oO) + (S)$ where 
$S = \oO + T + 2T + 3T = 2T \in E[2]$.  The factorization is 
determined by the automorphism group, and the image curve can 
be verified by elementary substitution.
\qed
\end{proof}

\begin{lemma}%[Weierstrass model]
\label{lem:Z4Weierstrass}
The $\ZZ/4\ZZ$-normal form is isomorphic to a curve in 
Weierstrass form $Y(Y + X)Z = X(X + c^{-1}Z)^2$.
The linear map $(X:Y:Z) = (X_1+X_2:X_2:c(X_0 + X_3))$ 
defines the isomorphism except at $\oO$. 
\end{lemma}

\begin{proof}
The existence of a linear map is implied by Kohel~\cite[Lemma~3]{Kohel}, 
and the exact form of this map can be easily verified.  
The exceptional divisor of the given rational map follows 
since $X_1 = X_2 = 0$ only meets the curve at $\oO$. 
\qed
\end{proof}

\begin{theorem}%[Properties of the split $\bbmu_4$-normal form] 
\label{thm:Mu4Symmetries}
Let $C/k$ be a curve in $\bbmu_4$-normal form over a field of characteristic~$2$.
\vspace{-2mm}
\begin{enumerate}
\item
The embedding divisor class of $C$ is $4(\oO)$.
\item
The point $T = (1:c:1:0)$ is a rational $4$-torsion point.
\item
The translation--by--$T$ and inverse morphisms are given by:
$$
\begin{array}{r}
\tau_T(X_0:X_1:X_2:X_3) = (X_3:X_0:X_1:X_2),\\ 
{[-1]}(X_0:X_1:X_2:X_3) = (X_0:X_3:X_2:X_1).
\end{array}
$$
\end{enumerate}
\end{theorem}

\begin{proof}
As in Theorem~\ref{thm:Z4Symmetries}, the correctness of 
automorphisms is implied by action on the points $\oO$ 
and $T$, and the relation $\tau_T^4 = 1$ shows that $T$ 
is $4$-torsion. 
Since the hyperplanes $X_i = 0$ cut out the divisors 
$4(T_{i+2})$ where $T_{k} = kT$, and $T$ is $4$-torsion, 
this gives the form of the embedding divisor class.
\qed
\end{proof}

\begin{lemma}%[Weierstrass model]
\label{lem:Mu4Weierstrass}
An elliptic curve in split $\bbmu_4$-normal form is isomorphic 
to the curve $Y(Y + X)Z = X(X + c^{-2}Z)^2$ in Weierstrass form.  
The linear map $(X:Y:Z) = (c(X_1+X_3):X_0+cX_1+X_2:c^4X_2)$
defines the isomorphism except at $\oO$. 
\end{lemma}

\begin{proof}
As above, the existence of a linear map is implied by 
Kohel~\cite[Lemma~3]{Kohel}, and the exact form of this map 
can be easily verified.  The exceptional divisor of the given 
rational map follows since $X_2 = 0$ only meets the curve at $\oO$. 
\qed
\end{proof}

\noindent{\bf Remark.} 
The rational maps of Lemma~\ref{lem:Z4Weierstrass} and~\ref{lem:Mu4Weierstrass} 
extend to isomorphisms, but there is no base-point free linear representative 
for these isomorphisms. 

\section{Isomorphisms with normal forms}
\label{sec:Isomorphisms}

Let $E_{c^2}$ denote an elliptic curve in $\ZZ/4\ZZ$-normal form and 
$C_{c}$ a curve in $\bbmu_4$-normal form. By Lemmas~\ref{lem:Z4Weierstrass} 
and~\ref{lem:Mu4Weierstrass}, the curves $E_{c^2}$ and $C_{c}$ are isomorphic, 
but by classification of their embedding divisor classes in 
Theorems~\ref{thm:Z4Symmetries} and~\ref{thm:Mu4Symmetries}, 
it follows from Lemma~\ref{lem:translation_linear} that there is no 
linear isomorphism between them.  In this section we obtain a 
classification of curves over with rational $4$-torsion point and 
make the isomorphism explicit for $E_{c^2}$ and $C_{c}$.

\begin{theorem}
\label{thm:normal_form_isomorphism}
Let $X/k$ be an elliptic curve over a field~$k$ of characteristic~$2$, 
with identity $\oO$ and $k$-rational point $T$ of order $4$, and suppose
that $c$ is an element of $k$ such that $j(X) = c^8$.
\begin{enumerate}
\item
There exists a unique isomorphism of $X$ over $k$ to a curve $E_{c^2}$ in 
$\ZZ/4\ZZ$-normal form sending $\oO$ to $(1:0:0:1)$ and $T$ to $(1:1:0:0)$.
\item
There exists a unique isomorphism of $X$ over $k$ to a curve $C_{c}$ in split 
$\bbmu_4$-normal form sending $\oO$ to $(c:1:0:1)$ and $T$ to $(1:c:1:0)$. 
\end{enumerate}  
If $X$ is embedded as a symmetric quartic model in $\PP^3$, then either the 
isomorphism of $X$ with $E_{c^2}$ or the isomorphism with $C_{c}$ is induced 
by a linear automophism of $\PP^3$. 
\end{theorem}

\begin{proof} 
The $j$-invariants of $E_{c^2}$ and $C_{c}$ are each $c^8$ ($\ne 0$ 
since $X$ is not supersingular by existence of a $2$-torsion point), 
which implies the existence of the isomorphisms over the algebraic 
closure. 
The rational $4$-torsion point $T$ fixes the quadratic twist, hence 
the isomorphism is defined over $k$. 
Since there is a unique $2$-torsion point $S = 2T$, the embedding 
divisor of $X$ in $\PP^3$ is either $3(\oO) + (S)$ or $4(\oO)$ by 
Lemma~\ref{lem:divisor_class}. 
In the former case, the isomorphism to $E_{c^2}$ is linear, and 
in the latter case the isomorphism to $C_{c}$ is linear by 
Lemma~\ref{lem:embedding_class}.
\qed \end{proof}

\noindent
The following theorem classifies the isomorphisms between $E_{c^2}$ and $C_{c}$.

\begin{theorem}
\label{thm:mu4-to-C4-isomorphism}
Let $C_{c}$ be an elliptic curve in split $\bbmu_4$-normal form and $E_{c^2}$ 
an elliptic curve in $\ZZ/4\ZZ$-normal form.  Then there exists an isomorphism 
$\iota : C_{c} \rightarrow E_{c^2}$ determined by the projections 
$$
\begin{array}{r@{\,}c@{\,}l}
\pi_1 \circ \iota((X_0:X_1:X_2:X_3)) & = &  
\left\{
\begin{array}{c}
(c X_0 : X_1 + X_3),\\ 
(X_1 + X_3 : c X_2),
\end{array}
\right.\\
\pi_2 \circ \iota((X_0:X_1:X_2:X_3)) & = & 
\left\{
\begin{array}{c}
(X_0 + X_2 : c X_1),\\ 
(c X_3 : X_0 + X_2).
\end{array}\right.
\end{array}
$$
The morphism to $E_{c^2}$ is recovered by composing $\pi_1 \times \pi_2$ with 
the skew-Segre embedding.  The inverse morphism is given by 
\begin{center}
\scalebox{0.90}{
$
\iota^{-1}(X_0:X_1:X_2:X_3) = 
\left\{
\begin{array}{c}
    %[ X0*X1 + X2*X3, c*X2^2, X0*X1 + c^2*X1*X2 + X2*X3, c*X1^2 ],
    ( X_0 X_1 + X_2 X_3 : c X_2^2 : X_0 X_1 + c^2 X_1 X_2 + X_2 X_3 : c X_1^2 ),\\
    %[ c*X0*X3, (X2 + X3)^2, c*X1*X2, (X0 + X1)^2 ],
    ( X_0 X_3 : (X_2 + X_3)^2 : X_1 X_2 : (X_0 + X_1)^2 ),\\
    %[ (X0 + X3)^2, c*X2*X3, (X1 + X2)^2, c*X0*X1 ],
    ( (X_0 + X_3)^2 : c X_2 X_3 : (X_1 + X_2)^2 : c X_0 X_1 ),\\
    %[ c*X3^2, X1*X2 + X0*X3 + c^2*X2*X3, c*X2^2, X1*X2 + X0*X3 ]
    ( c X_3^2 : X_0 X_3 + X_1 X_2 + c^2 X_2 X_3 :c X_2^2 : X_1 X_2 + X_0 X_3 ).
\end{array}
\right.
$
} % end scalebox
\end{center}
Neither $\iota$ nor its inverse can be represented by a projective linear 
transformation. 
\end{theorem}

\begin{proof} 
This correctness of this isomorphism can be verified explicitly 
(e.g.\ as implemented in Echidna~\cite{Echidna}). 
The nonexistence of a linear isomorphism is a consequence of 
Lemma~\ref{lem:translation_linear} and the classification of 
the embedding divisor classes in Theorems~\ref{thm:Z4Symmetries}
and~\ref{thm:Mu4Symmetries}.
\qed \end{proof}

\section{Addition law structure and efficient arithmetic}
\label{sec:AdditionLaws}

The interest in alternative models of elliptic curves has been driven by the 
simple form of their {\it addition laws} --- the polynomial maps which define 
the addition morphism $\mu: E \times E \rightarrow E$ as rational maps.  
In this section we determine bases of simple forms for the addition laws of 
the $\ZZ/4\ZZ$-normal form and of the $\bbmu_4$-normal form.

The verification that a system of putative addition laws determines 
a well-defined morphism can be verified symbolically. In particular we 
refer to the implementations of these models and their addition laws in 
Echidna~\cite{Echidna} (in the Magma~\cite{Magma} language) for a 
verification that the systems are consistent and define rational maps. 
The dimensions of the spaces of given bidegree are known {\it a~priori} 
by Kohel~\cite{Kohel}, as well as their completeness as morphisms.
By the Rigidity Theorem~\cite[Theorem 2.1]{Milne-AV}, a morphism $\mu$ 
of abelian varieties is the composition of a homomorphism and translation.
In order to verify that $\mu : E \times E \rightarrow E$ is the addition 
morphism, it suffices to check that the restrictions of $\mu$ to 
$E \times \{\oO\}$ and $\{\oO\} \times E$ agree with the restrictions 
of $\pi_1$ and $\pi_2$, respectively.  Similarly, for a particular 
addition law of bidegree $(2,2)$, the exceptional divisors, on which 
the polynomials of the addition law simultaneously vanish, are known 
by Lange and Ruppert~\cite{LangeRuppert} and the generalizations in 
Kohel~\cite{Kohel} to have components of the form 
$
\Delta_P = \{ (P+Q,Q) \;|\; Q \in E \}.
$
Consequently, as pointed out in Kohel~\cite{Kohel} (Corollary~11 and 
the Remark following Corollary~12), the exceptional divisors can be 
computed (usually by hand) by intersecting with $E \times \{\oO\}$.

\subsection*{Addition law structure for the $\ZZ/4\ZZ$-normal form}

\begin{theorem}
\label{thm:C4-addition-laws}
Let $E/k$, $\mathrm{char}(k) = 2$, be an elliptic curve in $\ZZ/4\ZZ$-normal 
form:
$$
(X_0 + X_1 + X_2 + X_3)^2 = c X_0 X_2 = c X_1 X_3.
$$
Bases for the bilinear addition law projections $\pi_1 \circ \mu$ and 
$\pi_2 \circ \mu$ are, respectively:
$$
\left\{
\begin{array}{l}
(X_0 Y_3 + X_2 Y_1,\ X_1 Y_0 + X_3 Y_2),\\
(X_1 Y_2 + X_3 Y_0,\ X_0 Y_1 + X_2 Y_3)
\end{array}
\right\}
\mbox{ and }
\left\{
\begin{array}{l}
(X_0 Y_0 + X_2 Y_2,\ X_1 Y_1 + X_3 Y_3),\\
(X_1 Y_3 + X_3 Y_1,\ X_0 Y_2 + X_2 Y_0)
\end{array} 
\right\}\cdot
$$
Addition laws of bidegree $(2,2)$ % given by polynomial maps 
are recovered by composition with the skew-Segre embedding   
$s((U_0:U_1),(V_0:V_1)) = (U_0 V_0:U_1 V_0:U_1 V_1:U_0 V_1)$. 
Each of these basis elements has an exceptional divisor of 
of the form $2\Delta_{nT}$ for some $0 \le n \le 3$. 
\end{theorem}

\begin{proof}
That the addition laws determine a well-defined morphism is verified 
symbolically.\!\footnote{%
In Echidna~\cite{Echidna}, the constructor is 
{\tt EllipticCurve\_C4\_NormalForm} after which {\tt AdditionMorphism} 
returns this morphism as a composition.}
The morphism is the addition morphism since the substitution 
$(Y_0,Y_1,Y_2,Y_3) = (1,0,0,1)$, gives the projection onto the 
first factor.  By symmetry of the spaces in $X_i$ and $Y_i$, 
the same holds for the second factor. 

The form of the exceptional divisor is verified by a similar
substitution.  For example, for the exceptional divisor
$
X_1 Y_2 + X_3 Y_0 = X_0 Y_1 + X_2 Y_3 = 0,
$
we intersect with $(Y_0,Y_1,Y_2,Y_3) = (1,0,0,1)$ to find 
$X_3 = X_2 = 0$, which defines the unique point $T = (1,1,0,0)$
with a multiplicity of 2, hence the exceptional divisor is 
$2\Delta_{T}$. The other exceptional divisors are determined 
similarly.
\qed 
\end{proof}

\noindent{\bf Remark.}
We observe that the entire space of addition laws of bidgree 
$(2,2)$ is independent of the curve parameters.  This is not 
a feature of the Edwards addition laws. 

\begin{corollary}
Addition of generic points on $E$ can be carried out in $12\Mul$. 
\end{corollary}

\begin{proof}
Since each of the pairs is equivalent under a permutation of the input 
variables it suffices to consider the first, which each require $4\Mul$.
Composition with the skew-Segre embedding requires an additional $4\Mul$, 
which yields the bound of $12\Mul$.
\qed
\end{proof}

\noindent
Evaluation of the addition forms along the diagonal yields the duplication 
formulas.

\begin{corollary}
Let $E = E_{c}$ be an elliptic curve in $\ZZ/4\ZZ$-normal form.  
The duplication morphism on $E$ is given by 
$$
\begin{array}{r@{\,}c@{\,}l}
\pi_1 \circ [2](X_0:X_1:X_2:X_3) & = & (X_0 X_3 + X_1 X_2 : X_0 X_1 + X_2 X_3),\\
\pi_2 \circ [2](X_0:X_1:X_2:X_3) & = & ((X_0 + X_2)^2 : (X_1 + X_3)^2),
\end{array}
$$
composed with the skew-Segre embedding. 
\ignore{
The resulting degree $4$ defining polynomials for the isogeny are unique, 
up to the defining ideal, and equivalent to 
$$
\begin{array}{l}
(
  X_0 X_3^3 + c X_0 X_1 X_3^2 + X_1 X_2^3 :
  X_0 X_1^3 + c X_0 X_1^2 X_3 + X_2^3 X_3 :\\
\multicolumn{1}{r}{\,
  X_0 X_1^3 + X_2^3 X_3 + c X_1 X_2 X_3^2 :
  X_0 X_3^3 + c X_1^2 X_2 X_3 + X_1 X_2^3
).}
\end{array}
$$
}
\end{corollary}

\noindent
This immediately gives the following complexity for duplication.

\begin{corollary}
Duplication on $E$ can be carried out in $7\Mul + 2\Sqr$. 
\end{corollary}

\begin{proof}
The pair $(X_0 X_3 + X_1 X_2, X_0 X_1 + X_2 X_3)$ can be computed 
with $3\Mul$ by exploiting the usual Karatsuba trick using 
the factorization of their sum:
$$
(X_0 X_3 + X_1 X_2) +  (X_0 X_1 + X_2 X_3) = (X_0 + X_2) (X_1 + X_3).
$$
After the two squarings, the remaining $4\Mul$ come from the 
Segre morphism.
\qed
\end{proof}

\subsection*{Addition law structure for the split $\bbmu_4$-normal form}

\begin{theorem}
\label{thm:mu4-addition-laws}
Let $C$ be an elliptic curve in split $\bbmu_4$-normal form:
$$
(X_0+X_2)^2 = c^2\,X_1 X_3,\ 
(X_1+X_3)^2 = c^2\,X_0 X_2.
$$
A basis for the space of addition laws of bidegree $(2,2)$ is given by: 
\begin{center}
\scalebox{0.80}{
$
\left\{
\begin{array}{@{}l@{}}
\big(
    (X_0 Y_0 + X_2 Y_2)^2,\,
    c (X_0 X_1 Y_0 Y_1 + X_2 X_3 Y_2 Y_3),\,
    (X_1 Y_1 + X_3 Y_3)^2,\,
    c (X_0 X_3 Y_0 Y_3 + X_1 X_2 Y_1 Y_2)\,
\big),\\
\big(
    c (X_0 X_1 Y_0 Y_3 + X_2 X_3 Y_1 Y_2),\,
    (X_1 Y_0 + X_3 Y_2)^2,\,
    c (X_0 X_3 Y_2 Y_3 + X_1 X_2 Y_0 Y_1),\,
    (X_0 Y_3 + X_2 Y_1)^2
\big),\\
\big(
    (X_3 Y_1 + X_1 Y_3)^2,\,
    c (X_0 X_3 Y_1 Y_2 + X_1 X_2 Y_0 Y_3),\,
    (X_0 Y_2 + X_2 Y_0)^2,\,
    c (X_0 X_1 Y_2 Y_3 + X_2 X_3 Y_0 Y_1)\,
\big),\\
\big(
    c (X_0 X_3 Y_0 Y_1 + X_1 X_2 Y_2 Y_3),\,
    (X_0 Y_1 + X_2 Y_3)^2,\,
    c (X_0 X_1 Y_1 Y_2 + X_2 X_3 Y_0 Y_3),\,
    (X_1 Y_2 + X_3 Y_0)^2
\big).
\end{array}
\right\}
$
} % end scalebox
\end{center}
The exceptional divisor of each addition law is of the form $4\Delta_{nT}$.
\end{theorem}

\begin{proof}
As for the $\ZZ/4\ZZ$-normal form the consistency of the addition laws 
is verified symbolically\footnote{%
The Echidna~\cite{Echidna} constructor is 
{\tt EllipticCurve\_Split\_Mu4\_NormalForm} after which 
{\tt AdditionMorphism} returns this morphism as a composition.}
and the space is known to have dimension four by Kohel~\cite{Kohel}.
Evaluation of the first addition law at $(Y_0,Y_1,Y_2,Y_3) 
= (c,1,0,1)$ gives
$$
(c^2 X_0^2,\, c^2 X_0 X_1,\, (X_1 + X_3)^2,\, c^2 X_0 X_3).
$$
Using $(X_1 + X_3)^2 = c^2 X_0 X_2$, after removing the common 
factor $c^2 X_0$, this agrees with projection to the first factor, 
and identifies the exceptional divisor $4\Delta_S$ where $S$ is the 
$2$-torsion point $(0:1:c:1)$ with $X_0 = 0$.
\qed
\end{proof}

\begin{corollary}
\label{cor:mu4-addition-complexity}
Addition of generic points on $C$ can be carried out in $7\Mul + 2\Sqr + 2\mul_c$. 
\end{corollary}

\begin{proof}
\begin{itemize}
\item Evaluate $(Z_0,Z_1,Z_2,Z_3) = (X_0 Y_0, X_1 Y_1, X_2 Y_2, X_3 Y_3)$ with $4\Mul$.
\item Evaluate $(X_0 Y_0 + X_2 Y_2)^2 = (Z_0 + Z_2)^2$ with $1\Sqr$.
\item Evaluate $(X_1 Y_1 + X_3 Y_3)^2 = (Z_1 + Z_3)^2$ with $1\Sqr$.
\item Evaluate $(X_0 Y_0 + X_2 Y_2)(X_1 Y_1 + X_3 Y_3) = (Z_0 + Z_2)(Z_1 + Z_3)$ followed by 
$$
\begin{array}{l}
X_0 X_1 Y_0 Y_1 + X_2 X_3 Y_2 Y_3 = Z_0 Z_1 + Z_2 Z_3\\
X_0 X_3 Y_0 Y_3 + X_1 X_2 Y_1 Y_2 = Z_0 Z_3 + Z_1 Z_2
\end{array}
$$
using $3\Mul$, exploiting the linear relation (following Karatsuba):
$$
(Z_0 + Z_2)(Z_1 + Z_3) = (Z_0 Z_1 + Z_2 Z_3) + (Z_0 Z_3 + Z_1 Z_2).
$$
\end{itemize}
After two scalar multiplications by $c$, we obtain $7\Mul + 2\Sqr + 2 \mul_c$ 
for the computation using the first addition law. \qed
\end{proof}

\noindent
Specializing this to the diagonal we find defining polynomials for duplication.

\begin{corollary}
\label{cor:mu4-duplication}
The duplication morphism on an elliptic curve $C$ in split 
$\bbmu_4$-normal form is given by    
$$
\begin{array}{r@{}l}
[2]&(X_0:X_1:X_2:X_3) = \\
&
(
  (X_0 + X_2)^4 :
c(X_0 X_1 + X_2 X_3)^2 : 
  (X_1 + X_3)^4 : 
c(X_0 X_3 + X_1 X_2)^2
).
\end{array}
$$
\end{corollary}

\noindent
This gives an obvious complexity bound of $3\Mul + 6\Sqr + 2\mul_c$ for 
duplication, however we note that along the curve we have the following 
equivalent expressions:
$$
\begin{array}{l}
c^2(X_0 X_1 + X_2 X_3)^2 = (X_0 + X_2)^4 + c^{-4}(X_1 + X_3)^4 + F^2,\\
c^2(X_0 X_3 + X_1 X_2)^2 = (X_0 + X_2)^4 + c^{-4}(X_1 + X_3)^4 + G^2,
\end{array}
$$ 
for $F = (X_0 + cX_3)(cX_1 + X_2)$ and $G = (X_0 + cX_1)(X_2 + cX_3)$,
and that 
$$
F + G = c(X_0 + X_2)(X_1 + X_3).
$$
This leads to a savings of $1\Mul + 1\Sqr$ from the naive analysis, at
the cost of extra multiplications by $c$.

\begin{corollary}
\label{cor:mu4-duplication-complexity}
Duplication on $C$ can be carried out in $2\Mul + 5\Sqr + 7\mul_c$.
\end{corollary}

\begin{proof}
We describe the evaluation of the forms of Corollary~\ref{cor:mu4-duplication}, 
using the equivalent expressions. 
Setting $(U,V,W) = ((X_0 + X_2)^2,(X_1+X_3)^2,(X_0+cX_1)^2)$, 
$$
G^2 = (U + c^2V + W)\,W \mbox{ and } F^2 = G^2 + c^2UV,
$$
from which the duplication formula can be expressed as: 
$$
( cU^2 : U^2 + c^{-4}V^2 + (U + c^2V + W)W + c^2UV : 
  cV^2 : U^2 + c^{-4}V^2 + (U + c^2V + W)W).
$$
We scale by $c^4$ to have only integral powers of $c$, which gives 
the   
\begin{itemize}
\item Evaluate $(U,V,W) = ((X_0 + X_2)^2,(X_1+X_3)^2,(X_0+cX_1)^2)$ with $3\Sqr + 1\mul_c$.
\item Evaluate $c^5(X_0 + X_2)^4 = c^5 U^2$ with $1\Sqr + 1\mul_c$, storing $U^2$.
\item Evaluate $c^5(X_1 + X_3)^4 = c^5 V^2$ with $1\Sqr + 1\mul_c$, storing $V^2$.
\item Evaluate $c^2 V$, $c^2 U V$, $(U + c^2V + W)\,W$ with $2\Mul + 1\mul_c$, then set
$$
\begin{array}{l}
c^4(X_0 + X_2)^4 + (X_1 + X_3)^4 = c^4 U^2 + V^2,\\
c^4 G^2 = c^4 (U + c^2 V + W)\,W,\\
c^4 F^2 = c^4 G^2 + c^6 U V,\\
\end{array}
$$
\end{itemize}
using $3\mul_c$, followed by additions.  This gives the asserted complexity.
\qed
\end{proof}

\noindent
{\bf Remark.} The triple $(U,V,W)$, up to scalars, can be identified with 
the variables $(A,B,C)$ of the EFD~\cite{EFD} in the improvement of 
Bernstein et al.~\cite{BinaryEdwards} to the duplication algorithm 
of Kim and Kim~\cite{KimKim} in ``extended L\'opez-Dahab coordinates'' 
with $a_2 = 0$.  In brief, the extended L\'opez-Dahab coordinates defines 
a curve
$
Y^2 = (X^2 + a_6) X Z,
$
in a $(1,2,1,2)$-weighted projective space %$\PP^3_{(1,2,1,2)}$ 
with coordinate functions $X$, $Y$, $Z$, $Z^2$. We embed this 
in a standard $\PP^3$, with embedding divisor class $4(\oO)$, 
by the map $(X^2,Y,XZ,Z^2)$.  By Lemma~\ref{lem:embedding_class}
this is linearly isomorphic to the curve $C$ in split 
$\bbmu_4$-normal form. One derives an equivalent complexity 
for duplication on this $\PP^3$ model, and duplication on $C$ 
differs only by the cost of scalar multiplications involved 
in the linear transformation to $C$.

We remark that this can be interpretted as a factorization of 
the duplication map as follows. Letting $D$ be the image of $C$ 
given by $(U,V,W)$ in $\PP^2$, the Kim and Kim algorithm can be 
expressed as a composition
$
C \xrightarrow{\ \varphi\ } D \xrightarrow{\ \psi\ } C
$
where $\varphi$ and $\psi$ are each of degree 2, with $\varphi$ 
purely inseparable and $\psi$ separable.  
The curve $D$ is a singular quartic curve in $\PP^2$, given by 
a well-chosen incomplete linear system.  The nodal singularities 
of $D$ are oriented such that the resolved points have the same 
image under $\psi$. The omission of a fourth basis element of 
the complete linear system allows one to save $1\Sqr$ in its 
computation.  

In order to best optimize the multiplications by scalars, we can 
apply a coordinate scaling. The split $\bbmu_4$-normal descends 
to a (non-split) $\bbmu_4$-normal form over any subfield 
containing the parameter $s = c^{-4}$, by renormalization 
of coordinates:
$$
s(X_1 + X_3)^2 + X_0 X_2,\ (X_0 + X_2)^2 = X_1 X_3.
$$
In this form the duplication polynomials require fewer multiplications 
by constants: 
$$
\begin{array}{rl}
    \big(\,(X_0 + X_2)^4 : & (X_0 + X_2)^4 + s^2(X_1 + X_3)^4 + (X_0 + X_3)^2(X_1 + X_2)^2 :\\
    s(X_1 + X_3)^4 : &
    (X_0 + X_2)^4 + s^2(X_1 + X_3)^4 + (X_0 + X_1)^2(X_2 + X_3)^2\,\big),

\end{array}
$$
yielding $2\Mul + 5\Sqr + 2\mul_s$. 

\subsection*{Addition law projections for the split $\mu_4$-normal form}

Let $C = C_{c}$ be an elliptic curve in $\bbmu_4$-normal form and 
$E = E_{c^2}$ be an elliptic curve in $\ZZ/4\ZZ$-normal form. 
In view of Theorem~\ref{thm:mu4-to-C4-isomorphism}, there is an explicit 
isomorphism $\iota: C \rightarrow E$, determined by the 
application of the skew-Segre embedding to the pair of projections
$\pi_i: C \rightarrow \PP^1$:
$$
\begin{array}{rcl}
\pi_1((X_0:X_1:X_2:X_3)) & = &  
\left\{
\begin{array}{c}
(c X_0 : X_1 + X_3),\\ 
(X_1 + X_3 : c X_2),
\end{array}
\right.\\
\pi_2((X_0:X_1:X_2:X_3)) & = & 
\left\{
\begin{array}{c}
(X_0 + X_2 : c X_1),\\ 
(c X_3 : X_0 + X_2).
\end{array}\right.
\end{array}
$$
The first projection $\pi_1$ determines a map to $C/\langle[-1]\rangle \isom \PP^1$,
and the second projection $\pi_2$ satisfies $\pi_2 \circ [-1] = \sigma \circ \pi_2$,
where $\sigma((U_0:U_1)) = (U_1:U_0)$.  As a consequence of the addition law structure 
of Theorem~\ref{thm:mu4-addition-laws}, the addition law projections $C \times C 
\rightarrow \PP^1$ associated to these projections take a particularly simple form.
\begin{corollary}
\label{cor:mu4-addition-law-projections}
If $\pi_i: C \rightarrow \PP^1$ are the projections defined above, the addition law 
projections $\pi_1 \circ \mu$ and $\pi_2 \circ \mu$ are respectively spanned by 
\ignore{
$$
\pi_1 \circ \mu((X_0:X_1:X_2:X_3),(Y_0:Y_1:Y_2:Y_3)) = 
\left\{
\begin{array}{c}
(X_0 Y_0 + X_2 Y_2 : X_1 Y_1 + X_3 Y_3),\\
(X_1 Y_3 + X_3 Y_1 : X_2 Y_0 + X_0 Y_2),
\end{array}
\right.
$$
$$
\pi_2 \circ \mu((X_0:X_1:X_2:X_3),(Y_0:Y_1:Y_2:Y_3)) = 
\left\{
\begin{array}{c}
(X_0 Y_3 + X_2 Y_1 : X_1 Y_0 + X_3 Y_2),\\
(X_1 Y_2 + X_3 Y_0 : X_0 Y_1 + X_2 Y_3).
\end{array}
\right.
$$
}
$$
\left\{
\begin{array}{c}
(X_0 Y_0 + X_2 Y_2, X_1 Y_1 + X_3 Y_3),\\
(X_1 Y_3 + X_3 Y_1, X_2 Y_0 + X_0 Y_2)
\end{array}
\right\}
\mbox{ and }
\left\{
\begin{array}{c}
(X_0 Y_3 + X_2 Y_1, X_1 Y_0 + X_3 Y_2),\\
(X_1 Y_2 + X_3 Y_0, X_0 Y_1 + X_2 Y_3)
\end{array}
\right\}\cdot
$$
\end{corollary}

\begin{proof}
The addition law projections can be verified in Echidna~\cite{Echidna}.
\qed
\end{proof}

The skew-Segre embedding of $\PP^1 \times \PP^1$ in $\PP^3$ induces a map 
to the isomorphic curve $E$ in $\ZZ/4\ZZ$-normal form, rather than 
the $\bbmu_4$-normal form.  These addition law projections play a central 
role in the study of the Kummer arithmetic in Section~\ref{sec:KummerCurves},
defined more naturally in terms of $E$.

\section{Kummer quotients and the Montgomery ladder}
\label{sec:KummerCurves}

For an abelian variety $A$, the quotient variety $\cK(A) = A/\{[\pm1]\}$ 
is called the Kummer variety of $A$.  We investigate explicit models for 
the Kummer curves $\cK(E)$ and $\cK(C)$ where $E = E_{c^2}$ and $C = C_{c}$ 
are isomorphic elliptic curves in $\ZZ/4\ZZ$-normal form and $\bbmu_4$-normal form, 
respectively.  The objective of this study is to obtain a Montgomery 
ladder~\cite{Montgomery} for efficient scalar multiplication on these curves.  
Such a Montgomery ladder was developed for Kummer curves (or {\it lines} 
since they are isomorphic to the projective line $\PP^1$) in characteristic~2
by Stam~\cite{Stam}.  More recently Gaudry and Lubicz~\cite{GaudryLubicz} 
developed efficient pseudo-addition natively on a Kummer line $\cK = \PP^1$
by means of theta identities.  Neither the method of Stam nor Gaudry and 
Lubicz provides recovery of points on the curve. 
We show that for fixed $P$, the morphism $E \rightarrow \cK \times \cK$ 
sending $Q$ to $(\overline{Q},\overline{Q-P})$, used for initialization 
of the Montgomery ladder, is in fact an isomorphism with its image. 
As a consequence we rederive the equations of Gaudry and Lubicz for 
pseudo-addition, together with an algorithm for point recovery. 
In addition, knowledge of the curve equation (in $\cK \times \cK$) 
permits the trade-off of a squaring for a multiplication 
by a constant depending on the base point $P$ 
(see Corollary~\ref{cor:Montgomery-complexity}).

\subsection*{Kummer curves}

We consider the structure of $\cK(E) = E/\{[\pm1]\}$ and $\cK(C) = C/\{[\pm1]\}$ 
for elliptic curves $E = E_{c^2}$ in $\ZZ/4\ZZ$-normal form and $C = C_{c}$ 
in split $\bbmu_4$-normal form, respectively.  
The former has a natural identification with $\PP^1$ equipped with the 
covering $\pi_1: E \rightarrow \cK(E)$, given by 
$$
(X_0:X_1:X_2:X_3) \mapsto 
\left\{
\begin{array}{l}
(X_0:X_1),\\
(X_3:X_2).
\end{array}
\right.
$$
The latter quotient has a plane model $\cK(C): Y^2 = c^2XZ$ in $\PP^2$ obtained 
by taking the $[-1]$-invariant basis $\{X_0,X_1+X_3,X_2\}$. 
For $E = E_{c^2}$ and $C = C_c$ as above, the isomorphism 
$\iota: C \rightarrow E$ of Theorem~\ref{thm:mu4-to-C4-isomorphism} 
induces an isomorphism $\iota: \cK(C) \rightarrow \cK(E) = \PP^1$ of 
Kummer curves given by 
$$
\iota(X:Y:Z) = \left\{
\begin{array}{l}
(c X : Y),\\
(Y : c Z),
\end{array}
\right.
$$
with inverse $(U_0:U_1) \mapsto (U_0^2 : c U_0 U_1 : U_1^2)$. 
Hereafter we fix this isomorphism, %mapping $\cK(C)$ to $\cK(E) = \PP^1$, 
and obtain the covering morphism $C \rightarrow \cK(E)$:
$$
(X_0:X_1:X_2:X_3) \mapsto 
\left\{
\begin{array}{l}
(cX_0:X_1+X_3),\\
(X_1+X_3:cX_2).
\end{array}
\right.
$$
We denote this common Kummer curve by $\cK$, to distinguish the curve 
with induced structure from the elliptic curve covering (by both $E$ 
and $C$) from $\PP^1$.  

\subsection*{Montgomery endomorphism}

The Kummer curve $\cK$ (of an arbitrary elliptic curve $E$)
no longer supports an addition morphism, however scalar multiplication 
$[n]$ is well-defined, since $[-1]$ commutes with $[n]$.
We investigate the general construction of the Montgomery ladder for 
the Kummer quotient.  For this purpose we define the {\it Montgomery endomorphism} 
$E \times E \rightarrow E \times E$:
$$
\left(
\begin{array}{cc}
2\ & 0 \\
1\ & 1 
\end{array}
\right)(Q,R) = (2Q,Q+R).
$$
In general this endomorphism, denoted $\varphi$, is not well-defined 
on $\cK \times \cK$.  Instead, for fixed $P \in E(k)$ we consider 
$$
\Delta_P = \{ (Q,R) \in E \times E \;|\; Q - R = P \} \isom E,
$$
and let $\cK(\Delta_P)$ be the image of $\Delta_P$ in $\cK \times \cK$,
which we call a {\it Kummer-oriented curve}.  
In what follows we develop algorithmically the following observations 
(see Theorems~\ref{thm:Z4-Kummer-oriented}, \ref{thm:mu4-Kummer-oriented},
and~\ref{thm:Kummer-oriented-Montgomery}):
\begin{enumerate}
\item
The morphism $\Delta_P \rightarrow \cK(\Delta_P)$ is an isomorphism 
for any $P \not\in E[2]$.
\item
The Montgomery endomorphism is well-defined on $\cK(\Delta_P)$. 
\end{enumerate}

By means of the elliptic curve structure on $\Delta_P$ determined 
by the isomorphism $E \rightarrow \Delta_P$ given by $Q \mapsto 
(Q,Q-P)$, the Montgomery endomorphism is the duplication morphism
(i.e.\ $\varphi(Q,Q-P) = (2Q,2Q-P)$).
On the other hand, the Montgomery endomorphism allows us to 
represent scalar multiplication on $P$ symmetrically as a sequence 
of compositions.  Precisely, we let $\varphi_0 = \varphi$, let
$\sigma$ be the involution $\sigma(Q,R) = (-R,-Q)$ of $\Delta_P$, 
which induces the exchange of factors on $\cK(\Delta_P)$, and 
set $\varphi_1 = \sigma \circ \varphi \circ \sigma$.  
For an integer $n$ with binary representation $n_r n_{r-1} \dots 
n_1 n_0$ we may compute $nP$ by the sequence
$$
\varphi_{n_0}\circ\varphi_{n_1}\cdots\circ\varphi_{n_{r-1}}(P,\oO) = ((n+1)P,nP),
$$
returning the second component. 
\ignore{
def MontgomeryScalar(v,n):
    phi0 = Matrix([[1,1],[0,2]])
    phi1 = Matrix([[2,0],[1,1]])
    phi = (phi0,phi1)
    t = ZZ(n).bits()
    t.reverse()
    for i in t:
        v = phi[i]*v
    return v

V = FreeModule(ZZ,2)
v = V([1,0])
for i in range(16):
    print '%2s == %2s' %(i, MontgomeryScalar(v,i)[1])
}

This composition representation for scalar multiplication on 
$E \times E$ is a double-and-always-add algorithm~\cite{Coron1999}, 
which provides a symmetry protection against side-channel 
attacks in cryptography (see Joye and Yen~\cite[Section~4]{JoyeYen}), 
but is inefficient due to insertion of redundant additions.  
When applied to $\cK(\Delta_P)$, on the other hand, this gives 
a (potentially) efficient algorithm, conjugate duplication, 
for carrying out scalar multiplication.
In view of this, $\cK(\Delta_P)$ should be thought of as a model 
oriented for carrying out efficient scalar multiplication on a 
fixed point $P$ in $E(k)$.

\subsection*{The Kummer-oriented curves $\cK(\Delta_P)$}

Let $E = E_{c^2}$ be a curve in $\ZZ/4\ZZ$-normal form, let 
$P = (t_0:t_1:t_2:t_3)$ be a fixed point in $E(k)$, and let 
$\cK(\Delta_P)$ be the Kummer-oriented curve in $\cK^2$, with 
coordinate functions $((U_0,U_1),(V_0,V_1))$. 

\begin{theorem}
\label{thm:Z4-Kummer-oriented}
The Kummer-oriented curve $\cK(\Delta_P)$ in $\cK^2$, for $P = (t_0:t_1:t_2:t_3)$,  
has defining equation
$$
t_0^2 (U_0 V_1 + U_1 V_0)^2 + t_1^2 (U_0 V_0 + U_1 V_1)^2 = c^2 t_0 t_1 U_0 U_1 V_0 V_1. 
$$
If $P$ is not a $2$-torsion point, the morphism $\kappa: 
E \rightarrow \cK(\Delta_P)$, defined by $Q \mapsto (\overline{Q},\overline{Q-P})$, 
is an isomorphism, given by 
$$
\begin{array}{r@{\,}c@{\,}c@{\,}c@{\,}l}
\pi_1 \circ \kappa (X_0:X_1:X_2:X_3) & = & (U_0:U_1) & = & 
\left\{
\begin{array}{c}
(X_0 : X_1),\\ 
(X_3 : X_2),
\end{array}
\right.\\[4mm]
\pi_2 \circ \kappa (X_0:X_1:X_2:X_3) & = & (V_0:V_1) & = & 
\left\{
\begin{array}{c}
(t_0 X_0 + t_2 X_2 : t_3 X_1 + t_1 X_3),\\
(t_1 X_1 + t_3 X_3 : t_2 X_0 + t_0 X_2),
\end{array}
\right.
\end{array}
$$
with inverse
$$
\begin{array}{r@{\,}c@{\,}l}
\pi_1 \circ \kappa^{-1}((U_0:U_1),(V_0:V_1)) & = & (U_0:U_1)\\[2mm]
\pi_2 \circ \kappa^{-1}((U_0:U_1),(V_0:V_1)) & = & 
\left\{
\begin{array}{l}
( t_1 U_0 V_0 + t_2 U_1 V_1 : t_0 U_0 V_1 + t_3 U_1 V_0 ),\\
( t_3 U_0 V_1 + t_0 U_1 V_0 : t_2 U_0 V_0 + t_1 U_1 V_1 ).
\end{array}
\right.
\end{array}
$$
\end{theorem}

\begin{proof} 
The form of $\kappa$ follows from the definition of the 
addition law. The equation for the image curve can be 
computed by taking resultants, and verified symbolically.
The composition of $\kappa$ with projection onto the 
first factor is the Kummer quotient of degree 2. 
However, for all $P$ not in $E[2]$, the inverse morphism 
induces a nontrivial involution 
$$
(\overline{Q},\overline{Q-P}) \longmapsto 
(\overline{-Q},\overline{-Q-P}) = 
(\overline{Q},\overline{Q+P})
$$ 
on $\cK(\Delta_P)$. Consequently the map to $\cK(\Delta_P)$ 
has degree one, and being nonsingular, gives an isomorphism.  
\qed \end{proof}

\noindent{\bf Remark.}
We observe that $\cK(\Delta_P) = \cK(\Delta_{-P})$ in $\cK^2$, 
but that a change of base point changes $\kappa$ by $[-1]$. 
\vspace{2mm}

\noindent
The isomorphism of $E_{c^2}$ with $C_{c}$ lets us derive the 
analogous result for curves in $\bbmu_4$-normal form.

\begin{theorem}
\label{thm:mu4-Kummer-oriented}
Let $C = C_c$ be an elliptic curve in split $\bbmu_4$-normal 
form with rational point $S = (s_0:s_1:s_2:s_3)$. 
The Kummer-oriented curve $\cK(\Delta_S)$ in $\cK^2$ is given by 
the equation
$$
s_0 (U_0 V_1 + U_1 V_0)^2 + s_2 (U_0 V_0 + U_1 V_1)^2 = c (s_1 + s_3) U_0 U_1 V_0 V_1. 
$$
If $S$ is not a $2$-torsion point, the morphism 
$\lambda: C \rightarrow \cK(\Delta_S)$ is an 
isomorphism, and defined by 
$$
\begin{array}{r@{\,}c@{\,}l}
\pi_1 \circ \lambda (X_0:X_1:X_2:X_3) & = & 
\left\{
\begin{array}{l}
(c X_0 : X_1 + X_3),\\
(X_1 + X_3 : c X_2),
\end{array}
\right.\\[4mm]
\pi_2 \circ \lambda (X_0:X_1:X_2:X_3) & = & 
\left\{
\begin{array}{l}
(s_0 X_0 + s_2 X_2 : s_1 X_1 + s_3 X_3),\\
(s_3 X_1 + s_1 X_3 : s_2 X_0 + s_0 X_2),
\end{array}
\right.
\end{array}
$$
with inverse $\lambda^{-1}((U_0:U_1),(V_0:V_1))$ equal to 
\ignore{
$$
\left\{
\begin{array}{@{}l}
(
  (s_1 + s_3) U_0^2 V_0 : 
  (s_0 U_0^2 + s_2 U_1^2) V_1 + c s_1 U_0 U_1 V_0 :
  (s_1 + s_3) U_1^2 V_0 : 
  (s_0 U_0^2 + s_2 U_1^2) V_1 + c s_3 U_0 U_1 V_0
),\\
(
  (s_1 + s_3) U_0^2 V_1 : 
  (s_2 U_0^2 + s_0 U_1^2) V_0 + c s_3 U_0 U_1 V_1 : 
  (s_1 + s_3) U_1^2 V_1 : 
  (s_2 U_0^2 + s_0 U_1^2) V_0 + c s_1 U_0 U_1 V_1
).
\end{array}
\right.
$$
}
\begin{center}
\scalebox{0.78}{
$
\left\{
\begin{array}{@{}l}
(
  (s_1 + s_3) U_0^2 V_0 : 
  (s_0 U_0^2 + s_2 U_1^2) V_1 + c s_1 U_0 U_1 V_0 :
  (s_1 + s_3) U_1^2 V_0 : 
  (s_0 U_0^2 + s_2 U_1^2) V_1 + c s_3 U_0 U_1 V_0
),\\
(
  (s_1 + s_3) U_0^2 V_1 : 
  (s_2 U_0^2 + s_0 U_1^2) V_0 + c s_3 U_0 U_1 V_1 : 
  (s_1 + s_3) U_1^2 V_1 : 
  (s_2 U_0^2 + s_0 U_1^2) V_0 + c s_1 U_0 U_1 V_1
).
\end{array}
\right.
$
} % end scalebox
\end{center}
\end{theorem}

\begin{proof} 
The isomorphism $\iota: E_{c^2} \rightarrow C_c$ sending $S$ 
to $T = (t_0:t_1:t_2:t_3)$ induces the isomorphism
$(s_0:s_1+s_3:s_2) = (t_0^2:ct_0t_1:t_1^2)$, by which we identify 
$\cK(\Delta_P)$ and $\cK(\Delta_S)$. 
The form of the morphism $\lambda$ follows from the form of 
projective addition laws of Corollary~\ref{cor:mu4-addition-law-projections},
and its inverse can be verified symbolically.
\qed \end{proof}

We now give explicit maps and complexity analysis for the 
Montgomery endomorphism $\varphi(Q,R) = (2Q,Q+R)$, on the 
Kummer quotient $\cK(\Delta_P)$ (or $\cK(\Delta_S)$ setting 
$(t_0:t_1) = (cs_0:s_1+s_3) = (s_1+s_3:cs_2)$).
In view of the application to scalar multiplication on $E$ 
or $C$, this gives an asymptotic complexity per bit of $n$, 
for computing $[n]P$. 

\begin{theorem}
\label{thm:Kummer-oriented-Montgomery}
The Montgomery endomorphism $\varphi$ is defined by:
$$
\begin{array}{l}
\pi_1 \circ \varphi((U_0:U_1),(V_0:V_1)) = (U_0^4 + U_1^4 : c U_0^2 U_1^2),\\
\pi_2 \circ \varphi((U_0:U_1),(V_0:V_1)) = ( t_1 (U_0 V_0 + U_1 V_1)^2 : t_0 (U_0 V_1 + U_1 V_0)^2).
\end{array}
$$
\end{theorem}

The sets of defining polynomials are well-defined everywhere 
and the following maps are projectively equivalent modulo the 
defining ideal: 
$$
\begin{array}{l}
( t_1 (U_0 V_0 + U_1 V_1)^2 : t_0 (U_0 V_1 + U_1 V_0)^2)\\
= ( t_0 (U_0 V_0 + U_1 V_1)^2 : t_1 (U_0 V_0 + U_1 V_1)^2 + c\,t_0 (U_0 V_0) (U_1 V_1) )\\
= ( t_0 (U_0 V_1 + U_1 V_0)^2 + c\,t_1 (U_0 V_1) (U_1 V_0) : t_1 (U_0 V_1 + U_1 V_0)^2 ).
\end{array}
$$
Assuming the point normalization with $t_0 = 1$ or $t_1 = 1$, this 
immediately gives the following corollary.

\begin{corollary}
\label{cor:Montgomery-complexity}
The Montgomery endomorphism on $\cK(\Delta_P)$ can be computed with 
$4\Mul + 5\Sqr + 1\mul_t + 1\mul_c$ or with 
$4\Mul + 4\Sqr + 1\mul_t + 2\mul_c$.
\end{corollary}

The formulas so obtained agree with those of Gaudry and Lubicz~\cite{GaudryLubicz}. 
The first complexity result agrees with theirs and the second obtains a trade-off 
of one $\mul_c$ for one $\Sqr$ using the explicit equation of $\cK(\Delta_P)$ 
in $\cK^2$. Finally, the isomorphisms of Theorems~\ref{thm:Z4-Kummer-oriented}
and~\ref{thm:mu4-Kummer-oriented} permit point recovery, hence scalar 
multiplication on the respective elliptic curves.

\section{Conclusion}
\label{sec:Conclusion}

We conclude with a tabulation of the best known complexity results 
for doubling and addition algorithms on projective curves (taking 
the best reported algorithm from the EFD~\cite{EFD}). We include 
the Hessian model, the only cubic curve model, for comparison.  
It covers only curves with a rational $3$-torsion point. Binary 
Edwards curves~\cite{BinaryEdwards} cover general ordinary curves, 
but the best complexity result we give here is for $d_1 = d_2$ 
which has a rational $4$-torsion point.
Similarly, the L\'opez-Dahab model with $a_2=0$ admits a rational 
$4$-torsion point, hence covers the same classes, but the fastest 
arithmetic is achieved on the quadratic twists with $a_2=1$. 
The results here for addition and duplication on $\bbmu_4$-normal 
form report the better result (in terms of constant multiplications 
$\mul$) for the non-split $\bbmu_4$ model (see the remark after 
Corollary~\ref{cor:mu4-duplication-complexity} and 
Corollary~\ref{cor:nonsplit-mu4-addition-complexity} in the appendix).
\begin{center}
\begin{tabular}{|@{\;}l|@{\;}l|@{\;}l|}
\hline 
Curve model    & Doubling & Addition \\
\hline 
$\ZZ/4\ZZ$-normal form & $7\Mul + 2\Sqr$         & $12\Mul$ \\
Hessian                & $6\Mul + 3\Sqr$         & $12\Mul$ \\
Binary Edwards         & $2\Mul + 5\Sqr + 2\mul$ & $16\Mul + 1\Sqr + 4\mul$ \\ 
L\'opez-Dahab ($a_2=0$)  & $2\Mul + 5\Sqr + 1\mul$ & $14\Mul + 3\Sqr$ \\
L\'opez-Dahab ($a_2=1$)  & $2\Mul + 4\Sqr + 2\mul$ & $13\Mul + 3\Sqr$ \\
$\bbmu_4$-normal form  & $2\Mul + 5\Sqr + 2\mul$ &\;\;$7\Mul + 2\Sqr$ \\
\hline
\end{tabular}
\end{center}
This provides for the best known addition algorithm combined with 
essentially optimal doubling.  We note that binary Edwards curves 
with $d_1 = d_2$ and the L\'opez-Dahab model with $a_2=0$ and have 
canonical projective embeddings in $\PP^3$ such that the 
transformation to $\bbmu_4$-normal form is linear, so that, 
conversely, these models can benefit from the efficient addition 
of the $\bbmu_4$-normal form.

\section*{Appendix}

By means of a renormalization of variables, the split $\bbmu_4$-normal 
form can be put in $\bbmu_4$-normal form
$
(X_0+X_2)^2 = X_1 X_3,\ s(X_1+X_3)^2 = X_0 X_2,
$
where $s = c^{-4}$. This form loses the elementary symmetry given by 
cyclic permutation of the coordinates, but by the Remark following 
Corollary~\ref{cor:mu4-duplication-complexity}, we are able to save 
on multiplications by scalars in duplication.  This renormalization 
gives the following addition laws 
(as a consequence of Theorem~\ref{thm:mu4-addition-laws}), and 
give an analogous savings for addition. 

\begin{theorem}
\label{thm:nonsplit-mu4-addition-laws}
Let $C$ be an elliptic curve in $\bbmu_4$-normal form:
A basis for the space of addition laws of bidegree $(2,2)$ is given by: 
\begin{center}
\scalebox{0.80}{
$
\left\{
\begin{array}{@{}l@{}}
\big(
    (X_0 Y_0 + X_2 Y_2)^2,\,
    X_0 X_1 Y_0 Y_1 + X_2 X_3 Y_2 Y_3,\,
    s(X_1 Y_1 + X_3 Y_3)^2,\,
    X_0 X_3 Y_0 Y_3 + X_1 X_2 Y_1 Y_2\,
\big),\\
\displaystyle
\big(
    X_0 X_1 Y_0 Y_3 + X_2 X_3 Y_1 Y_2,\,
    (X_1 Y_0 + X_3 Y_2)^2,\,
    X_0 X_3 Y_2 Y_3 + X_1 X_2 Y_0 Y_1,\,
    (X_0 Y_3 + X_2 Y_1)^2
\big),\\
\big(
    s (X_1 Y_3 + X_3 Y_1)^2,\,
    X_0 X_3 Y_1 Y_2 + X_1 X_2 Y_0 Y_3,\,
    (X_0 Y_2 + X_2 Y_0)^2,\,
    X_0 X_1 Y_2 Y_3 + X_2 X_3 Y_0 Y_1\,
\big),\\
\big(
    X_0 X_3 Y_0 Y_1 + X_1 X_2 Y_2 Y_3,\,
    (X_0 Y_1 + X_2 Y_3)^2,\,
    X_0 X_1 Y_1 Y_2 + X_2 X_3 Y_0 Y_3,\,
    (X_1 Y_2 + X_3 Y_0)^2
\big).
\end{array}
\right\}\!\cdot
$
} % end scalebox
\end{center}
\end{theorem}
The absence of the constant $s$ in the 2nd and 4th addition laws permits 
us to save the $2\mul$ in the computation of addition.
\begin{corollary}
\label{cor:nonsplit-mu4-addition-complexity}
Addition of generic points on $C$ can be carried out in $7\Mul + 2\Sqr$.
\end{corollary}

\begin{proof}
After evaluating $(Z_0,Z_1,Z_2,Z_3) = (X_0 Y_1, X_1 Y_2, X_2 Y_3, X_3 Y_0)$ 
in the last addition law, the algorithm follows that of 
Corollary~\ref{cor:mu4-addition-complexity}.
\qed
\end{proof}

\end{document}